\documentclass[10pt]{article}

\usepackage{geometry,graphicx,amssymb,amsmath,amsbsy,eucal,amsfonts,mathrsfs,amscd,bm,amsthm}
\usepackage[all]{xy}
\usepackage{tikz}
\usetikzlibrary{arrows,shapes,calc,snakes}
\usetikzlibrary{shapes,snakes}
\usepackage{tikz-3dplot}
\usetikzlibrary{intersections}
\usepackage{tkz-euclide}
\usetkzobj{all}
\usetikzlibrary{angles}
\usepackage{pgfplots}

\numberwithin{equation}{section}

\allowdisplaybreaks[3]

\newtheorem{theorem}{Theorem}[section]
\newtheorem{lemma}[theorem]{Lemma}

\newtheorem{remark}[theorem]{Remark}
\newtheorem{corollary}[theorem]{Corollary}
\newtheorem{proposition}[theorem]{Proposition}
\theoremstyle{definition}
\newtheorem{definition}[theorem]{Definition}

\newcommand{\R}{\mathbb{R}}
\newcommand{\CK}{\mathcal{C}_K}
\newcommand{\bla}{{\bld \lambda}}

\newcommand{\p}{{\partial}}

\newcommand{\nab}{\nabla}

\newcommand{\mct}{\mathcal{T}_h}

\newcommand{\dive}{{\ensuremath\mathop{\mathrm{div}\,}}}
\newcommand{\Div}{{\rm div}\,}

\newcommand{\pol}{\EuScript{P}}
\newcommand{\bpol}{\boldsymbol{\pol}}
\newcommand{\bld}[1]{\boldsymbol{#1}}

\newcommand{\bI}{\bld{I}}

\newcommand{\bb}{\bld{b}}

\newcommand{\bv}{\bld{v}}
\newcommand{\bw}{\bld{w}}
\newcommand{\bs}{\bld{s}}
\newcommand{\bp}{\bld{p}}
\newcommand{\bn}{\bld{n}}

\newcommand{\bV}{\bld{V}}

\newcommand{\btheta}{\bld{\theta}}

\newcommand{\bH}{\bld{H}}

\newcommand{\bbeta}{{\bm \beta}}

\newcommand{\bpsi}{\bm \psi}

\newcommand{\bc}{\bm c}
\newcommand{\bbR}{\mathbb{R}}

\newcommand{\MV}{\bV^\text{MR}}

\newcommand\drawtriangle[1][]{

\coordinate (C3) at (5,0);
\coordinate (C1) at (0,0);
\coordinate (C2) at (2.5,4);

\coordinate (E1) at ($(C1)!0.5!(C3)$);
\coordinate (E2) at ($(C1)!0.5!(C2)$);
\coordinate (E3) at ($(C2)!0.5!(C3)$);

\coordinate (B) at (7.5/3,4/3)  {};

\fill[lightgray!30] (C1)--(C2)--(C3)--cycle; 
\draw[-](C1)--(C2)--(C3)--(C1);

}

\newcommand{\drawtetAlt}{

\pgfmathsetmacro{\factor}{1/sqrt(2)};
\pgfmathsetmacro{\COney}{0.25}; \pgfmathsetmacro{\COnez}{-0.00}; \pgfmathsetmacro{\COnex}{ 1.45*\factor};
\pgfmathsetmacro{\CTwoy}{0}; \pgfmathsetmacro{\CTwoz}{1.25}; \pgfmathsetmacro{\CTwox}{ 0*\factor};
\pgfmathsetmacro{\CThreey}{-1}; \pgfmathsetmacro{\CThreez}{0}; \pgfmathsetmacro{\CThreex}{0*\factor};
\pgfmathsetmacro{\CFoury}{1}; \pgfmathsetmacro{\CFourz}{0}; \pgfmathsetmacro{\CFourx}{ 0*\factor};

\coordinate (C1) at (\COnex,\COney,\COnez);
\coordinate (C2) at (\CTwox,\CTwoy,\CTwoz);
\coordinate (C3) at (\CThreex,\CThreey,\CThreez);
\coordinate (C4) at (\CFourx,\CFoury,\CFourz);

\pgfmathsetmacro{\COneTwox}{\COnex-\CTwox};
\pgfmathsetmacro{\COneTwoy}{\COney-\CTwoy};
\pgfmathsetmacro{\COneTwoz}{\COnez-\CTwoz};

\pgfmathsetmacro{\COneThreex}{\COnex-\CThreex};
\pgfmathsetmacro{\COneThreey}{\COney-\CThreey};
\pgfmathsetmacro{\COneThreez}{\COnez-\CThreez};

\pgfmathsetmacro{\COneFourx}{\COnex-\CFourx};
\pgfmathsetmacro{\COneFoury}{\COney-\CFoury};
\pgfmathsetmacro{\COneFourz}{\COnez-\CFourz};

\pgfmathsetmacro{\CTwoThreex}{\CTwox-\CThreex};
\pgfmathsetmacro{\CTwoThreey}{\CTwoy-\CThreey};
\pgfmathsetmacro{\CTwoThreez}{\CTwoz-\CThreez};

\pgfmathsetmacro{\CTwoFourx}{\CTwox-\CFourx};
\pgfmathsetmacro{\CTwoFoury}{\CTwoy-\CFoury};
\pgfmathsetmacro{\CTwoFourz}{\CTwoz-\CFourz};

\pgfmathsetmacro{\CThreeFourx}{\CThreex-\CFourx};
\pgfmathsetmacro{\CThreeFoury}{\CThreey-\CFoury};
\pgfmathsetmacro{\CThreeFourz}{\CThreez-\CFourz};

\coordinate (T1) at ($(C1)!0.33333!(C2)$);  
\coordinate (T2) at ($(C2)!0.666666!(C3)$); 
\coordinate (B1) at ($(T1)!0.5!(T2)$);

\coordinate (T3) at ($(C2)!0.33333!(C3)$);
\coordinate (T4) at ($(C3)!0.666666!(C4)$);
\coordinate (B2) at ($(T3)!0.5!(T4)$);

\coordinate (T5) at ($(C3)!0.33333!(C4)$);
\coordinate (T6) at ($(C4)!0.666666!(C1)$);
\coordinate (B3) at ($(T5)!0.5!(T6)$);

\coordinate (T7) at ($(C4)!0.33333!(C1)$);
\coordinate (T8) at ($(C1)!0.666666!(C2)$);
\coordinate (B4) at ($(T7)!0.5!(T8)$);

\fill[-, fill=gray!30, opacity=.5] (C1)--(C2)--(C3); 
\draw[-, black!100] (C1) --(C2)--(C3)--cycle;
\fill[-, fill=gray!30, opacity=.5] (C1)--(C2)--(C4); 

\draw[-, black!100] (C1) --(C2)--(C4)--cycle;
\draw[dashed,gray!90] (C3)--(C4);

\coordinate (M1) at ($(C1)!0.5!(C2)$);
\coordinate (M2) at ($(C3)!0.5!(C4)$);
\coordinate (B) at  ($(M1)!0.5!(M2)$); 

\tdplotcrossprod(\COneTwox,\COneTwoy,\COneTwoz)(\COneFourx,\COneFoury,\COneFourz);
\pgfmathsetmacro{\s}{0.8/sqrt(\tdplotresx*\tdplotresx + \tdplotresy*\tdplotresy + \tdplotresz*\tdplotresz)}
\coordinate (NE1) at ($(B4)+(\s*\tdplotresx, -0.125,\s*\tdplotresz+0.25)$);
\draw[thick,->,black!100] (B4) -- (NE1);
\tdplotcrossprod(\COneTwox,\COneTwoy,\COneTwoz)(\CTwoThreex,\CTwoThreey,\CTwoThreez);
\pgfmathsetmacro{\st}{0.25/sqrt(\tdplotresx*\tdplotresx + \tdplotresy*\tdplotresy + \tdplotresz*\tdplotresz)}
\coordinate (NE2) at ($(B1)+(\st*\tdplotresx, \st*\tdplotresy, \st*\tdplotresz)$);
\draw[thick,->,black!100] (B1) -- (NE2);
\tdplotcrossprod(\COneThreex,\COneThreey,\COneThreez)(\CThreeFourx,\CThreeFoury,\CThreeFourz);
\pgfmathsetmacro{\s}{0.25/sqrt(\tdplotresx*\tdplotresx + \tdplotresy*\tdplotresy + \tdplotresz*\tdplotresz)}
\coordinate (NE3) at ($(B3)+(\s*\tdplotresx, \s*\tdplotresy, \s*\tdplotresz)$);
\draw[thick,->,black!100] (B3) -- (NE3);

\tdplotcrossprod(\CTwoThreex,\CTwoThreey,\CTwoThreez)(\CTwoFourx,\CTwoFoury,\CTwoFourz);
\pgfmathsetmacro{\s}{0.25/sqrt(\tdplotresx*\tdplotresx + \tdplotresy*\tdplotresy + \tdplotresz*\tdplotresz)}
\coordinate (NE4) at ($(B2)-(\s*\tdplotresx, \s*\tdplotresy, \s*\tdplotresz)$);
\draw[thick,->,gray!100] (B2) -- (NE4);

}

\newcommand\drawtet[1][]{

\pgfmathsetmacro{\factor}{sqrt(2)};
\pgfmathsetmacro{\alpha}{10}; 
\pgfmathsetmacro{\gamma}{38}; 
\pgfmathsetmacro{\tau}{0}; 

\pgfmathsetmacro{\cosA}{cos(\alpha)};
\pgfmathsetmacro{\sinA}{sin(\alpha)};

\pgfmathsetmacro{\cosG}{cos(\gamma)};
\pgfmathsetmacro{\sinG}{sin(\gamma)};

\pgfmathsetmacro{\cosT}{cos(\tau)};
\pgfmathsetmacro{\sinT}{sin(\tau)};

\coordinate (C1) at (\cosT*\cosA-\sinT*\sinG*\sinA-0.5*\cosT*\sinA*\factor-0.5*\sinT*\sinG*\cosA*\factor,-\sinT*\cosA-\cosT*\sinG*\sinA+0.5*\sinT*\sinA*\factor-0.5*\cosT*\sinG*\cosA*\factor,-\cosG*\sinA-0.5*\cosG*\cosA*\factor);
\coordinate (C2) at (-\cosT*\cosA+\sinT*\sinG*\sinA-0.5*\cosT*\sinA*\factor-0.5*\sinT*\sinG*\cosA*\factor,\sinT*\cosA+\cosT*\sinG*\sinA+0.5*\sinT*\sinA*\factor-0.5*\cosT*\sinG*\cosA*\factor,\cosG*\sinA-0.5*\cosG*\cosA*\factor);
\coordinate (C3) at (\sinT*\cosG+0.5*\cosT*\sinA*\factor+0.5*\sinT*\sinG*\cosA*\factor,\cosT*\cosG-0.5*\sinT*\sinA*\factor+0.5*\cosT*\sinG*\cosA*\factor,-\sinG+0.5*\cosG*\cosA*\factor);
\coordinate (C4) at (-\sinT*\cosG+0.5*\cosT*\sinA*\factor+0.5*\sinT*\sinG*\cosA*\factor,-\cosT*\cosG-0.5*-\sinT*\sinA*\factor+0.5*\cosT*\sinG*\cosA*\factor,\sinG+0.5*\cosG*\cosA*\factor);

\coordinate (T1) at ($(C1)!0.33333!(C2)$);  
\coordinate (T2) at ($(C2)!0.666666!(C3)$); 
\coordinate (B1) at ($(T1)!0.5!(T2)$);

\coordinate (T3) at ($(C2)!0.33333!(C3)$);
\coordinate (T4) at ($(C3)!0.666666!(C4)$);
\coordinate (B2) at ($(T3)!0.5!(T4)$);

\coordinate (T5) at ($(C3)!0.33333!(C4)$);
\coordinate (T6) at ($(C4)!0.666666!(C1)$);
\coordinate (B3) at ($(T5)!0.5!(T6)$);

\coordinate (T7) at ($(C4)!0.33333!(C1)$);
\coordinate (T8) at ($(C1)!0.666666!(C2)$);
\coordinate (B4) at ($(T7)!0.5!(T8)$);

\coordinate (T11) at ($(C1)!0.16666667!(C2)$);  
\coordinate (T12) at ($(C2)!0.83333333!(C3)$);  

\draw[-, fill=gray!30, opacity=.5] (C1) --(C3)--(C4)--cycle; 
\draw[-, black!100] (C1) --(C3)--(C4)--cycle;
\draw[-, fill=gray!30, opacity=.5] (C2) --(C3)--(C4)--cycle; 
\draw[-, black!100] (C2) --(C3)--(C4)--cycle;

\coordinate (M1) at ($(C1)!0.5!(C2)$);
\coordinate (M2) at ($(C3)!0.5!(C4)$);
\coordinate (B) at  ($(M1)!0.5!(M2)$); 

\draw[dashed,gray!90] (C1)--(C2);
}

\newcommand\drawnode[1]{
	\node[inner sep = 0pt, minimum size=3.5pt,fill = black!100,circle] (N1) at (#1) {};
}

\newcommand\drawhiddennode[1]{
	\node[inner sep = 0pt, minimum size=3.5pt,fill = black!50,circle] (N1) at (#1) {};
}

\newcommand\drawnormals[1]{
\draw[->,thick](2.5,0.)--(2.5,-0.4);
\draw[->,thick](1.25,2)--(0.9108006784,2.211999576);
\draw[->,thick](3.75,2)--(4.089199322,2.211999576);
}

\newenvironment{myproof}[1][\proofname]{%
  \proof[   Proof of #1]%
}{\endproof}

\begin{document}


\title{Inf-sup stable finite elements on barycentric refinements producing divergence--free approximations in arbitrary dimensions}

\author{Johnny Guzm\'an
\thanks{Division of Applied Mathematics, Brown University (johnny\_guzman@brown.edu)} 
\and Michael Neilan
\thanks{Department of Mathematics, University of Pittsburgh (neilan@pitt.edu)}}

\maketitle

\begin{abstract}
We construct
several stable finite element 
pairs for the Stokes problem 
on barycentric refinements
in arbitrary dimensions.
A key feature of the spaces
is that the divergence maps the discrete velocity
space onto the the discrete pressure space;
thus, when applied to models of incompressible flows,
the pairs yield divergence-free 
velocity approximations.
The key result is a local inf-sup stability
that holds for any dimension and for any
polynomial degree.  With this result,
we construct global divergence-free and stable pairs
in arbitrary dimension and for any polynomial degree.

\end{abstract}

\section{Introduction}

In the papers \cite{ArnoldQin92,Zhang04} it was shown that $\bpol^c_{k}-\pol_{k-1}$ is an inf-sup stable
and divergence-free pair 
on barycentric refine meshes in two and three dimensions if the polynomial size  $k$ is 
sufficiently large.
The strategy in the analysis, as shown by Zhang \cite{Zhang04}, is Stenberg's
macro-element technique \cite{Stenberg84}, where the crucial step is a local inf-sup stability
estimate on each macro tetrahedra/triangle.
Then, Bernardi-Raugel \cite{BernardiRaugel} finite elements are implicitly used to control  piecewise constants 
to prove  global inf-sup stability. 
 The use of the Bernardi-Raugel finite elements  is the reason one needs a restriction on $k$
 {to ensure global inf-sup stability}: For dimension $d=2$, $k \ge 2$ and for 
 dimension $d=3$, $k \ge 3$. 

One of our contributions in this paper is to extend the results in \cite{ArnoldQin92,Zhang04} to arbitrary space dimension $d\ge 2$.  
The key step, as in \cite{Zhang04}, is to prove a local inf-sup stability result.  By definition, a barycentric refinement takes a given mesh 
(which we call the macro mesh) and adds the barycenter of each simplex of the macro mesh to the set of vertices.  
We slightly generalize this construction by showing that
one can use any arbitrary point in the interior of each simplex (not just the barycenter), 
as long as the resulting mesh is shape regular.

We then derive several applications of the local inf-sup  stability
result. First, we with the help of the Bernardi-Raugel element, we show that  $\bpol^c_{k}-\pol_{k-1}$ is inf-sup on 
the refined mesh for $k \ge d$ (as was shown in \cite{ArnoldQin92,Zhang04} for $d=2,3$).  For lower order approximations  $1\le k<d$,
we use an idea introduced in \cite{GuzmanNeilan14_3D} 
and supplement the velocity space to obtain an inf-sup stable pair.
To this end, 
 we construct vector-valued, piecewise polynomial functions with respect to the refined mesh
 that have the
%
 same trace as the Bernardi-Raugel face bubbles on the skeleton of the macro mesh.
%
 The key difference,  compared to the Bernardi-Raugel face bubbles, is that the divergence of these functions
 are piecewise constant.  The existence of such finite element functions, which we call
 modified face bubbles, is guaranteed by the local inf-sup stability result.
Thus, we supplement $\bpol^c_k$ (for $1\le k<d$)  
locally with these modified face bubbles to get an inf-sup stable pair  on the refined mesh.

We also consider finite elements on the (unrefined) macro mesh. 
We show that $\bpol^c_1-\pol_0$ can be made stable by supplementing the velocity space $\bpol^c_1$ with the modified face bubbles.
Since the divergence of the modified face bubbles are piecewise constant, and thus contained in the pressure space, these finite element will produce divergence-free approximations 
for the Stokes and NSE problems.  These finite elements are developed in arbitrary dimension. The two-dimensional case seems to coincide with a pair of finite elements considered in \cite{ChristiansenHu}.

A final application is inspired by a finite element introduced in \cite{AlfeldSorokina}.  
There, an inf-sup stable and divergence-free macro element pair is constructed in two dimensions
with a piecewise linear, continuous pressure space.
Again, with the help of the modified face bubbles, we extend
these results to arbitrary dimension $d\ge 3$.

Advantages of divergence-free pairs for the Stokes/NSE problems
include, e.g., better stability and error estimates, and the enforcement
of several conservation laws and invariant properties. We refer
the reader to the survey article \cite{JohnLinke17} which highlights
the benefits of divergence-free pairs.
In addition to the above references, several other inf-sup stable pair of spaces that 
produce divergence-free approximations have been constructed.
These include high-order finite elements ($k\ge 2d$)
in two and three dimensions \cite{ScottVogelius85,FalkNeilan13,Neilan15,Zhang11},
as well as lower order pairs supplemented with rational functions \cite{GuzmanNeilan14,GuzmanNeilan14_3D}.
{Advantages of the proposed elements given here
are its relative simplicity and flexibility
with respect to dimension and polynomial 
degree.  The shape functions are piecewise polynomials
and therefore quadrature rules are immediately available.
We mention that the degrees of freedom
of our lowest-order element agree with those
given in \cite{ChristiansenHu}, where divergence-free
Stokes elements with respect to Powell-Sabin partitions
are considered (e.g., in three dimensions, every tetrahedron 
is split into $12$ sub-elements).
 However,
our elements are defined on a less stringent 
barycenter partition, which makes the implementation
simpler.
}

The paper is organized as follows. 
In the next section we introduce some notation used throughout the paper.  
Then, in Section \ref{sectionlocal} a local inf-sup stability result is proved. 
In Section \ref{sectionBernardiRaugel} the Bernardi-Raguel face bubble functions 
and its modification are introduced.   Then in Section \ref{sec:LowOrderNonRefine}
a low-order, divergence-free, and inf-sup pair on the macro mesh is constructed.
Finally, in  Section \ref{sectionrefined}, inf-sup and divergence-free stable finite elements 
are given on the refined mesh.

\section{Notation and Preliminaries}
We consider a family of  shape regular $\{\mct\}$ 
conforming simplicial triangulation of a polytope domain $\Omega\subset \mathbb{R}^d$. For each $K \in \mct$ let $x_K \in K$ 
be an interior point, and consider the refined triangulation  $\mct$  that subdivides each simplex $K$ into $(d+1)$ simplices by 
adjoining the vertices of $K$ with the new vertex $x_K$. The resulting refined triangulation is denoted by $\mct^r$. We assume that the point $x_K$ are chosen so that the family $\{\mct^r \}$ is also shape regular.  If $x_K$ is the barycenter of $K$, which is the most practical choice, then $\mct^r$ is the barycentric refinement of $\mct$. For any simplex $K$ we let $\pol_k(K)$ be the space of polynomials of degree at most $k$ defined on $K$. The vector-valued polynomials on a simplex are given by $\bpol_k(K) =[\pol(K)]^d$.  
We define
\begin{align*}
\mathring{\bpol}_k^c(\mathcal{S}):&=\{\bv\in \bH^1_0(\Omega):\ \bv|_K\in \bpol_k(K),\ \forall K\in \mathcal{S}\},\\
\mathring{\pol}_k(\mathcal{S}) :&= \{q\in L^2_0(\Omega):\ q|_K\in \pol_k(K),\ \forall K\in \mathcal{S}\},
\end{align*}
with either $\mathcal{S} = \mct$ or $\mathcal{S} = \mct^r$.

We 
denote by $K^r$ the triangulation of $K$:
\begin{equation*}
K^r= \{ T \in \mct^r: T \subset K \},
\end{equation*}
and will use the notation
\begin{align}
\mathring{\bpol}_k^c(K^r):&=\{\bv\in \bH^1_0(K):\ \bv|_T\in \bpol_k(T),\ \forall T\in K^r\},\label{polKr1}\\
\mathring{\pol}_k(K^r) :&= \{q\in L^2_0(K):\ q|_T\in \pol_k(T),\ \forall T\in K^r\}. \label{polKr2}
\end{align}

\section{Local inf-sup stability}\label{sectionlocal}
In this section we will prove that $\mathring{\bpol}_k^c(K^r)-{\mathring{\pol}_{k-1}(K^r)}$ is inf-sup stable for each $K \in \mct$. The result can be stated as follows.
%
\begin{theorem}\label{localthm}
Let $k \ge 1$. For any $K \in \mct$ and for any $p \in \mathring{\pol}_{k-1}(K^r)$, there exists $\bv \in \mathring{\bpol}_k^c(K^r)$ such that 
\begin{equation}\label{localthm1}
\dive \bv=p \quad \text{ on } K,
\end{equation}
with the bound
\begin{equation}\label{localthm2}
\|\bv\|_{H^1(K)} \le C \| p\|_{L^2(K)},
\end{equation}
where the constant $C>0$ only depends on $k$ and the shape regularity 
of $K^r$, but is independent of $p$.
\end{theorem}

\begin{figure}
\begin{center}
\includegraphics[scale=.8]{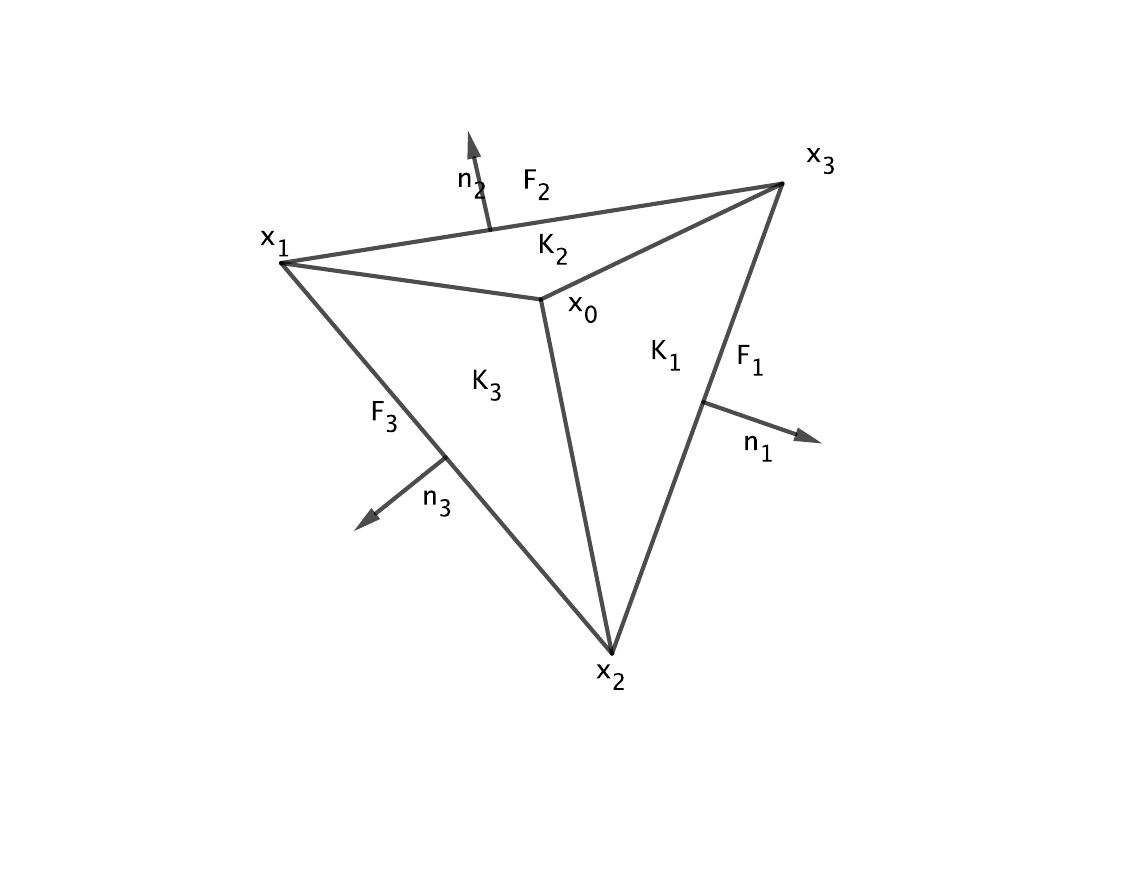}
\caption{Figure of macro triangle}\label{fig1}
\end{center}
\end{figure}

The proof of Theorem \ref{localthm} will follow from several lemmas.   
First, we will need some notation.  For $K\in \mct$,
denote by $S= \{x_1,\ldots,x_{d+1}\}$ 
the set of vertices of $K$, and let $x_0 = x_K$.
Then the refinement of $K$ is given by 
$K^r= \{ K_i\}_{1 \le i \le d+1}$,
where $K_i$ is the simplex with vertices  $\{x_0\}\cup S \backslash \{x_i\} $. 
We let $F_i$ be {the} $(d-1)$ dimensional face of $K$ opposite to $x_i$,
and let $\bn_i$ be the unit-normal to $F_i$ pointing out of $K$; see Figure \ref{fig1}. In addition to \eqref{polKr1}--\eqref{polKr2} we define
the polynomial spaces
\begin{alignat*}{2}
\pol_k(K^r):&=\{ v \in L^2(K):  v|_{K_i} \in  \pol_k(K_i),\ 
1 \le i \le d+1 \},\quad  &&\bpol_k(K^r):= [\pol_k(K^r)]^d,\\
\pol_k^c(K^r):&=\pol_k(K^r) \cap H^1(K), \quad &&\bpol_k^c(K^r):= [\pol_k^c(K^r)]^d.
\end{alignat*} 
Note that $\mathring{\bpol}_k^c(K^r)=\bpol_k^c(K^r) \cap \bld{H}_0^1(K)$ and $\mathring{\pol}_k(K^r)=\pol_k(K^r) \cap L_0^2(K)$.

For $0\le i\le d+1$, we
 let $\lambda_i \in \pol_1^c(K^r)$ be the continuous, piecewise linear function satisfying $\lambda_i( x_j)=\delta_{ij}$.  
We note that $\lambda_i$ vanishes on $K_i$.  For a multi-index $\alpha=(\alpha_1, \alpha_2, \ldots, \alpha_{d+1})$ we will use the notation
\begin{equation*}
\bla^\alpha= \lambda_1^{\alpha_1} \cdots \lambda_{d+1}^{\alpha_{d+1}}.
\end{equation*}
We note that 
\begin{equation}\label{gradl}
 \nabla( \bla^\alpha)= \sum_{j=1}^{d+1}   \alpha_j \bla^{\alpha-e_j}  \nabla \lambda_j, 
\end{equation}
where $\{e_j\}_{j=1}^{d+1}$ is canonical basis of $\R^{d+1}$. We also define 
\begin{equation*}
N(\alpha)=\{ i: \alpha_i > 0 \}.
\end{equation*}
Hence, $\bla^\alpha$ vanishes on  $K_i$ for all $i \in N(\alpha)$.

We let $h_K$ be the diameter of $K$, and let $\rho_K$ be the diameter of the largest ball inscribed in $K$. The shape regularity constant of $K$ is defined by
\begin{equation*}
C_K=\frac{h_K}{\rho_K}.
\end{equation*}
Analogously we let $C_{K_i}$ be the shape regularity constant of $K_i$ (for $i=1,2 \cdots,d+1$), 
and let $\CK=\max_{1 \le i \le d+1 }C_{K_i}$ denote the shape regularity constant of $K^r$. 
We see that $\CK$ is comparable to $C_K$ provided
$x_0$ is sufficiently far from $\p K$.

It is well known that for $ 0\le  i \le d+1$, we have  
\begin{equation*}
\|\nabla \lambda_i\|_{L^\infty(K)} \le \frac{C}{h_K},   
\end{equation*}
where $C$ only depends on  $\CK$. Hence, we also have
\begin{equation}\label{boundbla}
\|\nabla(\bla^\alpha)\|_{L^\infty(K)} \le \frac{C}{h_K},   
\end{equation}
where $C$ only depends on  $\CK$ and $|\alpha|$. 
In particular, we have
%
%
\begin{equation}\label{nabla0}
\nabla \lambda_0 |_{K_i}=-\frac{1}{h_i} \bn_i,
\end{equation}
where $h_i$ is the distance of $x_0$ to the $(d-1)$ dimensional hyperplane
that contains $F_i$. We note that
\begin{equation}\label{nabla0bound}
 h_i \le C h_K  \text{ for } 1 \le i \le d+1, 
\end{equation}
where $C$ depends only on $\CK$.

{Any $d$} unit normals
$\{\bn_{i_1}, \cdots, \bn_{i_d}\}$ where $1\le i_1 < i_2 < \cdots < i_d \le d+1$ are linearly
independent, and thus span $\bbR^d$.
Moreover, the $\ell^2$-norm
of the inverse of the
matrix  $A=[\bn_{i_1}, \bn_{i-2}, \cdots, \bn_{i_d}]^t$ depends on the shape regularity constant of $K$. More precisely, given  any $\bld{r} \in \R^d$ there exists a unique $\bs \in \R^d$ such that 
\begin{equation}\label{A}
A \bs =\bld{r}
\end{equation} 
with 
\begin{equation}\label{solvebs}
|\bs| \le C |\bld{r} |,
\end{equation}
where $C$ depends only on $C_K$.

We are interested in piecewise polynomials of the form 
\begin{equation*}
p=\lambda_0^\ell \sum_{|\alpha| = m} a_{\alpha} \bla^\alpha, \quad a_{\alpha} \in \pol_0(K^r),
\end{equation*}
and $a_{\alpha}|_{K_i}=0 \text{ for } i \in  N(\alpha)$.  A scaling argument shows that 
\begin{equation}\label{l2est}
\sum_{|\alpha|=m} \|a_{\alpha}\|_{L^2(K)}^2 \approx   \|p\|_{L^2(K)}^2,
\end{equation}
where the hidden constants depend
on the shape regularity of $K^r$, $m$ and $\ell$. 
In fact, the following decomposition for $\pol_s(K^r)$ holds. 
The result essentially follows from \eqref{l2est}, so we omit
the details.
%
\begin{lemma}
Every $p \in \pol_{s}(K^r)$ can be written uniquely as
\begin{equation}\label{aux1}
p= \sum_{\ell=0}^s p_\ell, 
\quad p_\ell=\lambda_0^\ell \sum_{|\alpha| = s-\ell} a_{\alpha} \bla^\alpha,
\end{equation}
where $a_{\alpha} \in  \pol_0(K^r)$, and $a_{\alpha}|_{K_i}=0 \text{ for } i \in  N(\alpha)$. Moreover, 
\begin{equation*}
\sum_{\ell=0}^s
\|p_\ell\|_{L^2(K)}^2 \le C \|p\|_{L^2(K)}^2, 
\end{equation*}
where the constant $C$ depends on the shape regularity constant of $K^r$ and $s$.
\end{lemma}

As mentioned earlier,
we will prove Theorem \ref{localthm}
 in several steps. 
First, we establish the following lemma.
%
\begin{lemma}\label{firststep}
Let $p= \lambda_0^\ell \sum_{|\alpha| = m} b_{\alpha} \bla^\alpha$  with $m \ge 1$,  $b_{\alpha} \in \pol_0(K^r)$, with $b_{\alpha}|_{K_i}=0$ for all $i \in N(\alpha)$. Then there exists $\bv \in \mathring{\bpol}_{\ell+m+1}^c(K^r)$ such that 
\begin{equation}\label{755}
\dive \bv =p +  q\text{ on } K,
\end{equation}
where $q$ is of the form $q=\lambda_0^{\ell+1} \sum_{|\alpha| = m-1} c_{\alpha} \bla^\alpha$ with  $c_{\alpha} \in \pol_0(K^r)$, with $c_{\alpha}|_{K_i}=0$ for all $i \in N(\alpha)$.
Moreover, the following bounds hold
\begin{equation}\label{l2v}
\|\bv \|_{H^1(K)} \le C \|p\|_{L^2(K)},\qquad \|q\|_{L^2(K)}\le C \|p\|_{L^2(K)}.
\end{equation} 
\end{lemma}
\begin{proof}
Since $|\alpha|=m \ge 1$,  the set 
$\{ i : i \notin N(\alpha)\}$ has at most $d$ elements.  Using the relation \eqref{nabla0}, solvability
of \eqref{A}, and the estimates \eqref{nabla0bound}, \eqref{solvebs},
we conclude that there exists  $\bs_{\alpha} \in \R^{d}$ such that
\begin{equation*}
(\ell+1) \bs_{\alpha} \cdot \nabla \lambda_0 |_{K_i}= b_{\alpha}|_{K_i} \text{ for  all } i \notin N(\alpha),
\end{equation*}
with the bound
\begin{equation}\label{aux143}
 |\bs_{\alpha}| \le  C  h_K \|b_{\alpha}\|_{L^\infty(K)},
\end{equation}
where the constant $C$ depends on $\CK$ and $\ell$.

Define $\bv= \lambda_0^{\ell+1} \sum_{|\alpha|=m} \bs_{\alpha} \bla^\alpha$ so that  $\bv \in \mathring{\bpol}_{\ell+m+1}^c(K^r)$, and 
\begin{alignat}{1}
\dive \bv= (\ell+1) \lambda_0^{\ell} \sum_{|\alpha|=m} (\bs_{\alpha} \cdot \nabla \lambda_0) \bla^\alpha+  \lambda_0^{\ell+1} \sum_{|\alpha|=m} \bs_{\alpha} \cdot \nabla( \bla^\alpha) =& p+ q,
\end{alignat}
where 
\begin{equation}\label{eqn:qBsForm}
q=\lambda_0^{\ell+1} \sum_{|\alpha|=m} \bs_{\alpha} \cdot \nabla( \bla^\alpha).
\end{equation}
Using \eqref{gradl} we see that we can write $q$ in the desired form. 

To obtain the bound \eqref{l2v},  we apply \eqref{boundbla} and an inverse estimate:
\begin{equation*}
\|\nabla \bv \|_{L^2(K)} \le C h_K^{d/2} \|\nabla \bv \|_{L^\infty(K)} \le C \, h_K^{d/2-1} \max_{ |\alpha|=m} |\bs_{\alpha}|.
\end{equation*}
Therefore by \eqref{aux143} we get 
\begin{equation*}
\|\nabla \bv \|_{L^2(K)} \le h_K^{d/2}\max_{ |\alpha|=m} \|b_\alpha\|_{L^\infty(K)}  \le  \max_{ |\alpha|=m} \|b_\alpha\|_{L^2(K)} \le  (\sum_{|\alpha|=m}  \|b_\alpha\|_{L^2(K)}^2 )^{1/2}.
\end{equation*}
Applying
 \eqref{l2est} and Friedrich's inequality to this last estimate,
 we obtain
\begin{equation*}
\|\bv\|_{H^1(K)}\le C \|\nabla \bv\|_{L^2(K)} \le C \| p\|_{L^2(K)}, 
\end{equation*}
which is the first bound in \eqref{l2v}.  Finally, 
the second bound in \eqref{l2v} follows from  {\eqref{755} and  the triangle inequality}.

\end{proof}

Using the previous result repeatedly, we prove the following result.
\begin{lemma}\label{lem:Repeated}
For any $p \in \pol_s(K^r)$,  there exists $\bv \in \mathring{\bpol}_{s+1}^c(K^r)$ such that 
\begin{equation}\label{rep}
\dive \bv= p+ b \lambda_0^s,
\end{equation}
where $b \in \pol_0(K^r)$. Moreover, the following bounds hold
\begin{equation}\label{lemmabound2}
\|\bv\|_{H^1(K)}\le C\|p\|_{L^2(K)},\qquad
\|b \lambda_0^s \|_{L^2(K)} \le C \|p\|_{L^2(K)},
\end{equation}
where the constant $C>0$ depends on $s$ and $\CK$.
\end{lemma}

\begin{proof}
Let $p\in \pol_s(K^r)$ be fixed
and write $p = \sum_{\ell=0}^s p_\ell$
as in \eqref{aux1}.
We then  conclude from Lemma \ref{firststep} that there exists $\bv_0 \in \mathring{\bpol}_{s+1}^c(K^r)$ 
and  $q_1=\lambda_0 \sum_{|\alpha|=s-1} c_{\alpha} \bla^\alpha$
such that 
\begin{equation*}
\dive \bv_0=p_0 + q_1,\qquad \|\bv_0\|_{H^1(K)}\le C\|p_0 \|_{L^2(K)},\qquad \|q_1\|_{L^2(K)}\le C\|p_0\|_{L^2(K)}.
\end{equation*}
%
Therefore, 
\begin{equation*}
p=\dive \bv_0+ (p_1-q_1) +p_2+ \cdots+ p_s.
\end{equation*}

Again, applying Lemma \ref{firststep} there exists $\bv_1 \in \mathring{\bpol}_{s+1}^c(K^r)$  so that 
\begin{equation*}
\dive \bv_1=(p_1-q_1)+ q_2,
\end{equation*}
where $q_2=\lambda_0^2 \sum_{|\alpha|=s-2} c_{\alpha} \bla^\alpha$ with the corresponding bounds. We can then continue this process 
to construct $\bv_2, \ldots, \bv_{s-1} \subset \bpol_{s+1}^c(K^r)$, and we set 
\begin{equation*}
\bv= \bv_0+ \cdots+ \bv_{s-1}.
\end{equation*}
The identity \eqref{rep} as well as the estimates \eqref{lemmabound2} are immediate.
\end{proof}

We can now prove Theorem \ref{localthm}.
\begin{myproof}[Theorem \ref{localthm}]
Let $p \in \mathring{\pol}_{k-1}(K^r)$. By Lemma \ref{lem:Repeated}, there exists $\bv_1 \in \mathring{\bpol}_{k}^c(K^r)$ such that 
\begin{equation*}
\dive \bv_1= p+ b \lambda_0^{k-1}\ \text{ for some } b\in \pol_0(K^r),
\end{equation*}
with the bounds
\begin{equation}\label{bv1}
\|\bv_1\|_{H^1(K)} \le C \|p\|_{L^2(K)},\qquad \| b \lambda_0^{k-1}\|_{L^2(K)} \le C \|p\|_{L^2(K)}.
\end{equation}

Since $p,\dive  \bv \in \mathring{\pol}_{k-1}(K^r)$ we have that 
\begin{equation}\label{aux307}
b \lambda_0^{k-1} \in \mathring{\pol}_{k-1}(K^r).
\end{equation}

Using \eqref{nabla0} and the solvability
of problem \eqref{A}, there exists $\bs\in \bbR^d$
such that
\begin{equation*}
k \bs \cdot \nabla \lambda_0|_{K_i}=b|_{K_i} \quad \text{ for all } 2 \le i \le d+1,
\end{equation*}
Set $\bv_2 = \bs \lambda_0^k$. From \eqref{solvebs}
and scaling, we have
\begin{equation}\label{bv2}
\|\bv_2\|_{H^1(K)} \le C \| b \lambda_0^{k-1}\|_{L^2(K)}.
\end{equation}
Moreover, 
\begin{align*}
\Div \bv_2|_{K_i} = k \lambda_0^{k-1} \bs \cdot \nab \lambda_0|_{K_i} = b \lambda_0^{k-1}|_{K_i},\qquad \text{for all }2\le i\le d+1.
\end{align*}
Since $\Div \bv_2 - b\lambda_0^{k-1}\in \mathring{\pol}_{k-1}(K^r)$, we have
\begin{align*}
0 = \int_K \big(\Div \bv_2 - b\lambda_0^{k-1}\big) = \int_{K_1} \big(\Div \bv_2 - b\lambda_0^{k-1}\big) = \int_{K_1} \lambda_0^{k-1}(\bs \cdot \nab \lambda_0 - b).
\end{align*}
Therefore $\bs\cdot \nab \lambda_0|_{K_1} = b|_{K_1}$, implying that $\Div \bv_2|_{K_1} = b\lambda_0^{k-1}|_{K_1}$ and so $\Div \bv_2 = b\lambda_0^{k-1}$ on $K$.
%
We  then set  $\bv=\bv_1+ \bv_2$. Then \eqref{localthm1} holds
and the 
bound \eqref{localthm2} follows from \eqref{bv1} and \eqref{bv2}. 
\end{myproof}

\section{The Bernardi-Raugel bubble {and} its modification}\label{sectionBernardiRaugel}
In this section we recall the\\ Bernardi-Raugel face bubbles (cf.~\cite{BernardiRaugel}) and summarize 
their stability properties.  Then, using Theorem \ref{localthm}, we  propose a modification of these 
bubble functions such that the resulting vector fields have constant divergence.

Recall, that for a simplex $K\in \mct$, the vertices are denoted
by $\{x_1,x_2,\ldots,x_{d+1}\}$, and that $F_i$ is the $(d-1)$-dimensional
face of $K$ opposite to $x_i$ with outward unit normal $\bn_i$.
We denote by $\mu_i\in \pol_1(K)$ the barycentric coordinates of $K$, i.e., $\mu_i(x_j)= \delta_{ij}$.
We define  scalar face bubbles as 
\begin{align*}
B_i := \mathop{\prod_{j=1}^n}_{j\neq i} \mu_j,\quad  \text{ for } 1 \le i \le d+1.
\end{align*}
The Bernardi-Raugel face bubbles are given as 
\begin{equation*}
\bb_i:=B_i \bn_i \quad  \text{ for } 1 \le i \le d+1.
\end{equation*}
We note that $\bb_i \in \bpol_d(K)$. 

We define the local Bernardi-Raugel bubble space as follows:
\begin{equation*}
\bV^\text{BR}(K):= \text{span} \{ \bb_1, \ldots, \bb_{d+1} \}.
\end{equation*}
The corresponding global space is given by
\begin{equation*}
\bV_h^\text{BR}:=\{ \bv \in \bld{H}_0^1(\Omega): \bv|_K \in \bV^\text{BR}(K), \text{ for all } K \in \mct \}.
\end{equation*}
Note that this space does not have the piecewise linear functions as the Bernardi-Raugel space has, but it is well known that the linear functions are only needed for approximation (not for stability).

The following result is well known (or can easily be proven); see \cite{BernardiRaugel}.
%
\begin{proposition}\label{bernardiraugel}
For any $p \in \mathring{\pol}_0(\mct)$, there exists a  $\bv \in \bV_h^\text{BR}$ so that 
\begin{equation*}
\int_{K} \dive \bv=\int_{K} p \qquad \text{ for all } K \in \mct, 
\end{equation*}
with the bound
\begin{equation*}
\|\bv\|_{H^1(\Omega)} \le C \| p \|_{L^2(\Omega)}.  
\end{equation*}
\end{proposition}

In the next result, we modify the function $\bb_i$ 
so that the resulting vector-valued function has constant divergence.

\begin{proposition}\label{prop:Bubbles}
There exists $\bbeta_i\in \bpol_d^c(K^r)$ such that
\begin{align}\label{eqn:bviDef}
\bbeta_i|_{\p K} = \bb_i|_{\p K},\quad \Div \bbeta_i\in \pol_0(K),
\end{align}
with the bound 
\begin{equation*}
\|\bbeta_i \|_{H^1(K)} \le C\|\bb_i \|_{H^1(K)}.
\end{equation*}

\end{proposition}
\begin{proof}
Set
\begin{align*}
g_i = \dive \bb_i - \frac{1}{|K|} \int_K \dive \bb_i\in \mathring{\pol}_{d-1}(K) \subset \mathring{\pol}_{d-1}(K^r).
\end{align*}
By Theorem \ref{localthm}, there exists $\bw_i\in \mathring{\bpol}^c_d(K^r)$ such that 
\begin{align}\label{eqn:bwiDef}
\dive \bw_i = g_i \quad \text{ on } K,\qquad \|\bw_i\|_{H^1(K)}\le C\|g_i\|_{L^2(K)}.
\end{align}
The function $\bbeta_i :=\bb_i - \bw_i$ then satisfies \eqref{eqn:bviDef}
since $\bbeta_i|_{\p K} = \bb_i|_{\p K}$ and $\dive \bbeta_i = \dive \bb_i - \dive \bw_i = |K|^{-1}\int_K \Div \bb_i$.
The stability estimate follows from \eqref{eqn:bwiDef} and the  bound $\|g_i\|_{L^2(K)}\le C\|\bb_i\|_{H^1(K)}$.
\end{proof}

We let $\{ \bbeta_1, \bbeta_2, \ldots, \bbeta_{d+1}\} \subset \bpol_d^c(K^r)$ 
be functions satisfying the conditions in 
Proposition \ref{prop:Bubbles} (note that they are not necessarily unique for $d\ge 3$, but we 
have fixed them). We call these functions the modified face bubbles. 
We define the local finite element space of these functions as follows:
\begin{equation*}
\MV(K):= \text{span} \{ \bbeta_1, \ldots, \bbeta_{d+1} \}.
\end{equation*}
%
\begin{lemma}\label{lem:MVDOFs}
A function $\bv \in \MV$ is uniquely determined by 
\begin{equation}\label{dofmbubbleLocal}
\int_{F_i} \bv \cdot \bn_i\quad  \text{ for } 1 \le i \le d+1. 
\end{equation}
\end{lemma}
\begin{proof}
Let $\bv\in \MV(K)$, 
and write $\bv = \sum_{i=1}^{d+1} a_i \bbeta_i$ for some $a_i\in \bbR$.
Suppose that $\bv$ vanishes on the degrees of freedom \eqref{dofmbubbleLocal}.
Then, since $\bbeta_i|_{\p K\backslash F_i} = \bb_i|_{\p K\backslash F_i} = 0$, we have
\begin{align*}
0 = \int_{F_i} \bv\cdot \bn_i = a_i \int_{F_i} \bb_i \cdot \bn_i = a_i \int_{F_i} B_i.
\end{align*}
Since $B_i>0$ on $F_i$, we conclude that $a_i=0$ and $\bv\equiv 0$.
The dimension of $\MV$ is clearly $(d+1)$, and therefore
we conclude that \eqref{dofmbubbleLocal} uniquely determines
a function in $\MV$.
%
%
\end{proof}
We  define the global space:
\begin{equation*}
\MV_h:=\{ \bv \in \bld{H}_0^1(\Omega): \bv|_K \in \MV(K), \text{ for all } K \in \mct \}.
\end{equation*}
Lemma \ref{lem:MVDOFs} shows that the
degrees of freedom for $\bv \in \bV_h^\text{MF}(K)$ are
\begin{equation}\label{dofmbubble}
 \int_{F} \bv\cdot \bn\quad  \text{  for all interior $(d-1)$-dimensional faces } F \text{ of } \mct. 
\end{equation}
%
\begin{theorem}\label{modifiedbubble}
It holds, $\dive \bV_h^\text{MF} \subset \mathring{\pol}_0(\mct)$. Moreover, 
for any $p \in \mathring{\pol}_0(\mct)$  there exists a  $\bv \in \bV_h^\text{MF}$ so that 
\begin{equation}\label{vp}
\dive \bv=p \qquad \text{ in } \Omega, 
\end{equation}
with the bound
\begin{equation*}
\|\bv\|_{H^1(\Omega)} \le C \| p \|_{L^2(\Omega)}.  
\end{equation*}
\end{theorem}

\begin{proof}
First, by \eqref{eqn:bviDef} we see that $\dive \MV_h \subset \mathring{\pol}_0(\mct)$. 

For a fixed $p \in  \mathring{\pol}_0(\mct)$, there exists $\bw \in \bld{H}_0^1(\Omega)$ such that (see for example \cite{Duran}).  
\begin{equation*}
\dive \bw=p\quad  \text{ in } \Omega, 
\end{equation*}
with the bound
\begin{equation*}
\|\bw\|_{H^1(\Omega)} \le C \| p \|_{L^2(\Omega)}.  
\end{equation*}
We then define $\bv \in \bV_h^\text{MF}$ by
\begin{equation*}
\int_F \bv \cdot \bn=\int_F \bw \cdot \bn \quad  \text{  for all interior $(d-1)$-dimensional faces } F \text{ of } \mct. 
\end{equation*}
A simple scaling argument gets 
\begin{equation*}
\| \bv\|_{H^1(\Omega)} \le C \|\bw\|_{H^1(\Omega)} \le  \| p \|_{L^2(\Omega)}.
\end{equation*}
Moreover, an application of the divergence theorem shows
\begin{equation*}
\int_K \dive \bv=\int_K \dive \bw=\int_K p \text{ for all } K \in \mct.
\end{equation*}
Since $\dive \bv, p  \in \mathring{\pol}_0(\mct)$, this proves \eqref{vp}. 
\end{proof}

\section{A low  order inf-sup stable pair on $\mct$}\label{sec:LowOrderNonRefine}
With the modified face bubble spaces,
we now provide several inf-sup stable and divergence-free pairs
applicable to incompressible flows. First, let
us recall the definitions of an inf-sup stable pair
and a divergence-free pair.
%
\begin{definition}\label{def:InfSup}
A pair of spaces $\bV_h-M_h$, with $\bV_h \subset \bld{H}_0^1(\Omega)$ and  $M_h \in L_0^2(\Omega)$ is {\em inf-sup stable} if there exists a constant $\gamma >0$ such that 
\begin{equation*}
 0< \gamma \le  \inf_{ 0 \neq q \in M_h}  \sup_{0 \neq \bv \in \bV_h} \frac{\int_{\Omega} q \, \dive \bv }{ \|\bv\|_{H^1(\Omega)} \|p\|_{L^2(\Omega)} }.
\end{equation*}
The pair is said to be a {\em divergence-free pair} if
$\dive \bV_h \subset M_h$.
\end{definition}

In this section, we give 
an example of a pair defined on the macro mesh $\mct$ satisfying
the two conditions in Definition \ref{def:InfSup}.
Note that Theorem \ref{modifiedbubble} shows 
that  $\bV_h^{\text{MF}}- \mathring{\pol}(\mct)$ is such a pair, i.e., 
it is a inf-sup stable and divergence-free pair.
However, $\bV_h^{\text{MF}}$ does not have good approximation properties;
thus, 
we supplement the bubble space with $\mathring{\bpol}_1^c(\mct)$. The following corollary 
to Theorem \ref{modifiedbubble} is immediate. 

\begin{corollary}\label{cor:LowOrderPairNonR}
The pair $\mathring{\bpol}_1^c(\mct)+ \MV_h- \mathring{\pol}_0(\mct)$ is a divergence free, inf-sup stable pair.  
\end{corollary}

\begin{remark}
The two-dimensional case of Corollary \ref{cor:LowOrderPairNonR}, although not written in this exact form, has appeared in \cite{ChristiansenHu}.
\end{remark}

\begin{remark}
It follows from Lemma \ref{lem:MVDOFs} and properties of the modified 
bubble functions that the degrees of freedom for $\bv\in \mathring{\bpol}_1^c(\mct)+ \MV_h$
are the same as the Bernardi-Raugel finite element space \cite{BernardiRaugel}, that is, the degrees of freedom are
\eqref{dofmbubble} and  function evaluation at all interior vertices of $\mct$; 
see Figure \ref{fig:DOFs3}.
\end{remark}

\begin{figure}
\begin{center}
\begin{tikzpicture}[scale=0.7]

\drawtriangle

\foreach \i in {1,2,3}
{	
	\draw[-, opacity=.5,dashed]  (C\i)--(B);
	\drawnode{C\i}
	\draw[thick,->](E\i)--($(E\i)!0.2!90:(C\i)$); 
}

\end{tikzpicture}
\tdplotsetmaincoords{70}{130}
\begin{tikzpicture}[tdplot_main_coords, scale=2.1]

\drawtetAlt

\foreach \i in {1,2,3,4}
{	
	\drawnode{C\i};
	\draw[-,dashed, opacity=.5]  (C\i)--(B);
}

\end{tikzpicture}
\caption{Degrees of freedom for
the velocity space in
Corollary \ref{cor:LowOrderPairNonR}
in two dimensions (left) and three dimensions (right).
Solid circles indicate function evaluation and arrows
indicate normal component evaluation.
}\label{fig:DOFs3}
\end{center}
\end{figure}
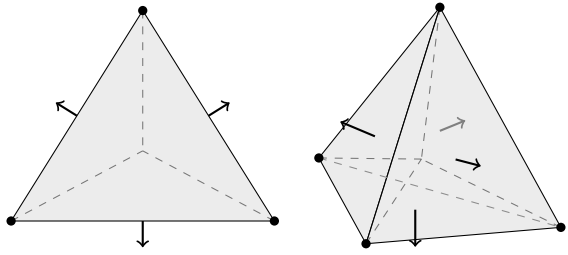

\section{Inf-sup stable pair of spaces on $\mct^r$}\label{sectionrefined}

In this section, we apply Theorem \ref{modifiedbubble}
to construct divergence-free and inf-sup stable pairs
on the refined mesh $\mct^r$. 
To do so, 
we will use the following result repeatedly. 
\begin{proposition}\label{lem:LowerOrderFramework}
Let $k\ge 1$, and suppose that $\bV_h\subset \bH^1_0(\Omega)$
satisfies $\mathring{\bpol}_k(\mct^r)\subset \bV_h$
and $\bV_h-\mathring{\pol}_0(\mct)$ is inf-sup stable.
Then $\bV_h-\mathring{\pol}_{k-1}(\mct^r)$ is inf-sup stable.
\end{proposition}

\begin{proof}
Let $q \in \mathring{\pol}_{k-1}(\mct^r)$, and define   $\bar{q}$ to be its $L^2$-projection onto $\mathring{\pol}_0(\mct)$, i.e.,
\begin{align*}
\bar{q}|_K = \int_K q\, dx\qquad \forall K\in \mct.
\end{align*}

By Theorem {\ref{localthm}} there exists $\bw\in \mathring{\bpol}^c_k(\mct^r) \subset \bV_h$ so that 
\begin{equation*}
\dive w= q-\bar{q} \quad \text{ on } \Omega,
\end{equation*}
\begin{equation}\label{wC}
\|w\|_{H^1(\Omega)} \le  C \|q-\bar{q}\|_{L^2(\Omega)}.
\end{equation}
Hence, 
\begin{equation*}
\|q-\bar{q}\|_{L^2(\Omega)} = \int_{\Omega} (q-\bar{q}) (q-\bar{q})= \int_{\Omega} q (q-\bar{q}) =\int_{\Omega} q\, \dive \bw.
\end{equation*}

Applying \eqref{wC}, we find
\begin{equation*}
\|q-\bar{q}\|_{L^2(\Omega)}^2 = \|w\|_{H^1(\Omega)}  \frac{ \int_{\Omega} q\, \dive \bw}{\|w\|_{H^1(\Omega)}} \le   C \|q-\bar{q}\|_{L^2(\Omega)} \sup_{0 \neq \bv \in \bV_h} \frac{\int_{\Omega} q \, \dive \bv }{ \|\bv\|_{H^1(\Omega)}},
\end{equation*}
and therefore, 
\begin{equation*}
\|q-\bar{q}\|_{L^2(\Omega)} \le   C \sup_{0 \neq \bv \in \bV_h} \frac{\int_{\Omega} q \, \dive \bv }{ \|\bv\|_{H^1(\Omega)}}.
\end{equation*}

By {the} hypothesis, there is a constant $\gamma_0 >0$ such that 
\begin{equation*}
 \gamma_0 \|\bar{q}\|_{L^2(\Omega)}  \le  \sup_{0 \neq \bv \in \bV_h} \frac{\int_{\Omega} \bar{q} \, \dive \bv }{ \|\bv\|_{H^1(\Omega)}} \le \sup_{0 \neq \bv \in \bV_h} \frac{\int_{\Omega} q \, \dive \bv }{ \|\bv\|_{H^1(\Omega)}}+ \|q-\bar{q}\|_{L^2(\Omega)} .
\end{equation*}
Hence, using the last two inequalities we get
\begin{equation*}
\|q\|_{L^2(\Omega)} \le (\|q-\bar{q}\|_{L^2(\Omega)} + \|\bar{q}\|_{L^2(\Omega)})  \le (C(\frac{1}{\gamma_0}+1)+  \frac{1}{\gamma_0})  \sup_{0 \neq \bv \in \bV_h} \frac{\int_{\Omega} {q} \, \dive \bv }{ \|\bv\|_{H^1(\Omega)}}.
\end{equation*}
This proves the result. 
\end{proof}

\subsection{Higher-order Elements with discontinuous pressures}
Using Proposition \ref{lem:LowerOrderFramework}, we now show that the pair  
$\mathring{\bpol}_k^c(\mct^r)-\mathring{\pol}_{k-1}(\mct^r)$ for $k \ge d$ is inf-sup stable. 
\begin{corollary}
The pair $\mathring{\bpol}_k^c(\mct^r)-\mathring{\pol}_{k-1}(\mct^r)$ with $k\ge d$ is a divergence-free, inf-sup stable pair.
\end{corollary}
\begin{proof}
Since $\bV_h^{\text{BR}} \subset \mathring{\bpol}_k^c(\mct^r)$ for $k \ge d$ and the fact that $\bV_h^{\text{BR}}-\mathring{\pol}_{0}(\mct)$ is inf-sup stable
(cf.~Proposition \ref{bernardiraugel}), the corollary follows by applying Proposition \ref{lem:LowerOrderFramework}.
\end{proof}

In order to establish that  $\mathring{\bpol}_k^c(\mct^r)-\mathring{\pol}_{k-1}(\mct^r)$  is 
inf-sup stable we used that $\bV_h^{\text{BR}}-\mathring{\pol}_{0}(\mct)$ is inf-sup stable;
in other words, the inclusion $\bV_h^{\text{BR}} \subset \mathring{\bpol}_k^c(\mct)$ implies 
that $\mathring{\bpol}_k^c(\mct)-\mathring{\pol}_{0}(\mct)$  is inf-sup stable. An interesting fact is that the converse is true. 

\begin{theorem}
The pair $\mathring{\bpol}_k^c(\mct^r)-\mathring{\pol}_{k-1}(\mct^r)$ is inf-sup stable if and only if  
$\mathring{\bpol}_k^c(\mct)-\mathring{\pol}_{0}(\mct)$ is inf-sup stable. 
\end{theorem}

\begin{proof}
Assume $\mathring{\bpol}_k^c(\mct)-\mathring{\pol}_{0}(\mct)$ is inf-sup stable.  
Then by the inclusion $\mathring{\bpol}_k^c(\mct) \subset \mathring{\bpol}_k^c(\mct^r)$,
the pair $\mathring{\bpol}_k^c(\mct^r)-\mathring{\pol}_{0}(\mct)$ is inf-sup stable. 
Thus, $\mathring{\bpol}_k^c(\mct^r)-\mathring{\pol}_{k-1}(\mct^r)$ is inf-sup stable by {Proposition} \ref{lem:LowerOrderFramework}.

Now suppose that $\mathring{\bpol}_k^c(\mct^r)-\mathring{\pol}_{k-1}(\mct^r)$ is inf-sup stable. 
Let $q \in \mathring{\pol}_{0}(\mct)$.  Due to the inclusion
$\mathring{\pol}_0(\mct) \subset \mathring{\pol}_{k-1}(\mct^r)$,  there exist $\gamma>0$ such that
\begin{equation*}
\gamma \|q\|_{L^2(\Omega)} \le \sup_{0 \neq \bv \in \mathring{\bpol}_k^c(\mct^r)} \frac{\int_{\Omega} q \, \dive \bv }{ \|\bv\|_{H^1(\Omega)} }.
\end{equation*} 
Let $\bI_h:[C(\bar{\Omega})]^d \rightarrow \mathring{\bpol}_k^c(\mct)$
be the canonical (nodal) interpolant.  
We then have
\begin{equation}\label{partial}
(\bI_h \bv-\bv)|_{\partial K}=0 \quad \text{ for all } K\in \mct, \text{ for all } \bv \in \mathring{\bpol}_k^c(\mct^r).
\end{equation}
Moreover, 
\begin{equation}\label{bound223}
\|\bI_h \bv \|_{H^1(\Omega)} \le C \|\bv \|_{H^1(\Omega)} \quad \text{ for all } \bv \in \mathring{\bpol}_k^c(\mct^r).
\end{equation}

By  \eqref{partial} and the divergence theorem we get that 
\begin{equation*} 
\int_{\Omega} q \, \dive \bv =\int_{\Omega} q \, \dive \bI_h \bv\quad  \text{ for all }  \bv \in \mathring{\bpol}_k^c(\mct^r) .
\end{equation*}
Hence, by \eqref{bound223},
\begin{equation*}
\gamma \|q\|_{L^2(\Omega)} \le \sup_{0 \neq \bv \in \mathring{\bpol}_k^c(\mct^r)} \frac{\int_{\Omega} q \, \dive \bI_h \bv }{ \|\bv\|_{H^1(\Omega)} } \le 
C^{-1} \sup_{0 \neq \bv \in \mathring{\bpol}_k^c(\mct^r)} \frac{\int_{\Omega} q \, \dive \bI_h \bv }{ \|\bI_h\bv\|_{H^1(\Omega)} }
\le C^{-1}  \sup_{0 \neq \bw \in \mathring{\bpol}_k^c(\mct)} \frac{\int_{\Omega} q \, \dive \bw }{ \|\bw\|_{H^1(\Omega)}}.
\end{equation*}
Therefore, $\mathring{\bpol}_k^c(\mct)-\mathring{\pol}_{0}(\mct)$ is inf-sup stable.

\end{proof}

\subsection{Low-order elements with discontinuous pressures}
For $k <d$,
we can always augment $\mathring{\bpol}_k^c(\mct^r)$ with $\bV_h^{\text{BR}}$ and it will lead to an inf-sup stable pair. 
However, the resulting pair will not be a divergence-free pair.  Therefore, we instead supplement  $\mathring{\bpol}_k^c(\mct^r)$
with $\MV_h$.

\begin{corollary}\label{cor:LowOrder}
Let $1 \le k < d$. The pair $\mathring{\bpol}_k^c(\mct^r)+\bV_h^{\text{MF}}-\mathring{\pol}_{k-1}(\mct^r)$ is a divergence-free, inf-sup stable pair. 
\end{corollary}
\begin{proof}
The result  follows from Proposition \ref{lem:LowerOrderFramework} and Lemma \ref{modifiedbubble}.
\end{proof} 
\begin{remark}
It follows from Proposition \ref{prop:Bubbles} and Lemma \ref{lem:MVDOFs}
that the degrees of freedom for  $\mathring{\bpol}_k^c(\mct^r)+\bV_h^{\text{MF}}$ ($k<d$) are the canonical degrees of freedom 
of $\mathring{\bpol}_k^c(\mct^r)$ plus the degrees of freedom \eqref{dofmbubble}; see Figure \ref{fig:DOFs1}.
\end{remark}

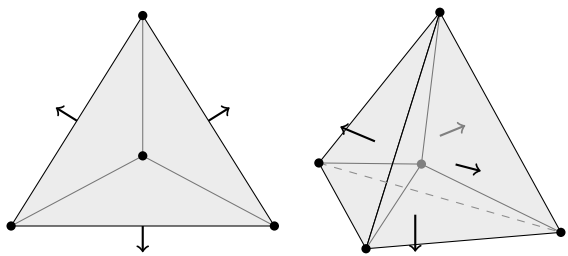
\begin{figure}
\begin{center}
\begin{tikzpicture}[scale=0.7]

\drawtriangle

\path let \p1 = (C1) in coordinate (b) at (2*\x1,\y1/2);

\foreach \i in {1,2,3}
{	
	\draw[-, opacity=.5]  (C\i)--(B); 
	\drawnode{C\i}; 
	\draw[thick,->](E\i)--($(E\i)!0.2!90:(C\i)$); 
}
\drawnode{B} 

\end{tikzpicture}
%
\tdplotsetmaincoords{70}{130}
\begin{tikzpicture}[tdplot_main_coords, scale=2.1]

\drawtetAlt

\foreach \i in {1,2,3,4}
{	
	\drawnode{C\i};
	\draw[-, opacity=.5]  (C\i)--(B);
}
\drawhiddennode{B}
\end{tikzpicture}

\caption{
Degrees of freedom for
the lowest order velocity space in
Corollary \ref{cor:LowOrder}
in two dimensions (left) and three dimensions (right).
Solid circles indicate function evaluation and arrows
indicate normal component evaluation.
}\label{fig:DOFs1}
\end{center}
\end{figure}

%

\subsection{Low-order Stokes pairs with continuous pressure}
In this section we, in some sense, generalize the inf-sup stable pair of spaces found  in \cite{AlfeldSorokina} to higher dimensions. 
In the paper \cite{AlfeldSorokina}, the pressure space is 
the space of continuous, piecewise linear polynomials with respect
to the refined triangulation:
%
\begin{equation*}
W_h^R= \mathring{\pol}_1(\mct^r)\cap H^1(\Omega).
\end{equation*}
Their velocity space is given by $\mathring{\bpol}_2^c(\mct^r) \cap \bH^1(\dive; \Omega)$, where
\begin{equation*}
\bH^1(\dive;S):= \{ \bv \in \bld{H}^1(S): \dive \bv \in \bld{H}^1(S)\}.
\end{equation*}
It is shown in \cite{AlfeldSorokina}, that this pair of spaces is inf-sup stable 
in two dimensions. It is clearly a divergence-free pair.

To generalize these results to higher dimensions it seems necessary to supplement the velocity space 
with the modified face bubbles given in Proposition \ref{prop:Bubbles}. The following result defines the local space
and a unisolvent set of degrees of freedom.
\begin{theorem}\label{thm:VRKrDOFsSmooth}
Define, for $d\ge 3$,
\begin{align*}
\bV_R(K):=\bpol_2(K^r)\cap \bH^1(\Div;K)\oplus {\rm span}\{\bbeta_i\}_{i=1}^{d+1}.
\end{align*}
Then a function $\bv\in \bV_R(K)$ is uniquely determined by the values
\begin{subequations}
\label{eqn:SmoothDOFs}
\begin{alignat}{2}
\label{eqn:SmoothDOFs1}
&\bv(x_i),\ \Div \bv(x_i)\quad &&\text{for all vertices } x_i \text{ of } K,\\
\label{eqn:SmoothDOFs2}
&\int_{e_i} \bv,\quad &&\text{for all one-dimensional edges $e_i$},\\
\label{eqn:SmoothDOFs3}
&\int_{F_i} \bv\cdot \bn_i,\quad &&\text{for all $(d-1)$-dimensional faces $F_i$}.
\end{alignat}
\end{subequations}
\end{theorem}
Before we prove this result, we note that in the case $d=2$ (which is not considered in this Theorem), 
the degrees of freedom \eqref{eqn:SmoothDOFs2} would contain the  degrees of freedom \eqref{eqn:SmoothDOFs3}. 
Therefore, in the case $d=2$, one simply has to eliminate the functions that give rise to  \eqref{eqn:SmoothDOFs3}, 
which are $\bbeta_1, \bbeta_2, \bbeta_3$; see \cite{AlfeldSorokina} for details.

\begin{myproof}[Theorem \ref{thm:VRKrDOFsSmooth}]
The constraint $\bv\in \bH^1(\Div;K)$ for $\bv\in \bpol^c_2(K^r)$
represents $(d+1)^2-(d+2)$ equations. Therefore
\begin{align*}
\dim \bV_R(K) 
&\ge \dim \bpol^c_2(K^r) - \big((d+1)^2 - (d+2)\big)+(d+1)\\
& = \frac12 (d+1)(d^2+2d+4).
\end{align*}
On the other hand, the number of degrees of freedom given is
\begin{align*}
d(d+1)+(d+1)+\frac{d^2}2 (d+1)+(d+1) = \frac12 (d+1)(d^2+2d+4).
\end{align*}

Now suppose that $\bv\in \bV_R(K)$ vanishes
on the degrees of freedom \eqref{eqn:SmoothDOFs}, and write $\bv = \bv_0+\bs$,
where $\bv_0\in \bpol_2^c(K^r)$
and $\bs = \sum_{i=1}^{d+1} c_i \bbeta_i$ for some $c_i\in \bbR$.
Then, since $\bs|_{F_i} = \bb_i$ on each $({d}-1)$-dimensional face,
we conclude  that
\begin{align*}
\bv_0(x_i) = 0,\qquad \int_{e_i} \bv_0 = 0
\end{align*}
for all vertices $x_i$ and edges $e_i$ of $K$.
These conditions imply that $\bv_0 = 0$ on $\p K$.
Therefore we have
\begin{align*}
0  = \int_{F_i} \bv\cdot \bn_i = \int_{F_i} \bs\cdot \bn_i = c_i \int_{F_i} \bb_i\cdot \bn_i.
\end{align*}
Since $\bb_i\cdot \bn_i>0$ on $F_i$, we obtain that $c_i=0$, and so $\bs\equiv 0$
and $\bv = \bv_0\in \bpol_2^c(K^r)$. Moreover, because $\Div \bv$ restricted
to a $({d}-1)$-dimensional face is a linear polynomial, we conclude
from the condition $\Div \bv(x_i)=0$ that $\Div \bv$ vanishes on $\p K$ as well.

Since $\bv$ vanishes on $\p K$ we can write  $\bv = \lambda_0 \bp$ (see Section \ref{sectionlocal} for definition of $\lambda_0$) for some $\bp\in \bpol^c_1(K^r)$.
We then find that 
\begin{align*}
0 = \Div \bv = \nab \lambda_0 \cdot \bp + \lambda_0\Div \bp = \nab \lambda_0 \cdot \bp\quad \text{on }\p K.
\end{align*}
The gradient of $\lambda_0$ restricted to $K_i$ is parallel to the outward unit normal
of the face $\p K_i\cap \p K$, and so we conclude that $\bp \cdot \bn=0$ on $\p K$.
This implies, since $\bp$ is continuous, that $\bp$ vanishes at the vertices of $K$.
But since $\bp$ is piecewise linear, we obtain that $\bp|_{\p K}=0$, and so $\bv = \bc \lambda_0^2$
for some $\bc \in \bbR^d$. However, it is easy to see that
$\Div \bv = 2\lambda_0 \bc \cdot \nab \lambda_0$ is only continuous if $\bc\equiv 0$.
Thus, $\bv \equiv 0$, and so the degrees of freedom are unisolvent on $\bV_R(K^r)$.
\end{myproof}

\begin{remark}
Note that $\Div \bbeta_i \in \pol_0(K) \subset \pol_1(K^r)\cap H^1(K)$.
Therefore $\bV_R(K^r)\subset \bH^1(\Div;K)$.
\end{remark}

\begin{remark}
If $\bv\in \bV_R(K^r)$ vanishes at the degrees of freedom
restricted to one face, then we can argue as in the proof of Theorem
 \ref{thm:VRKrDOFsSmooth} that $\bv =0$
and $\Div \bv=0$ on that face. Thus, the degrees
of freedom induce an $\bH^1(\Div;\Omega)$--conforming
finite element space.
\end{remark}

The local spaces and degrees of freedom lead
to the global finite element spaces
\begin{align*}
\bV_h^R = \{\bv\in  \bH^1_0(\Omega)\cap \bH^1(\Div;\Omega):\ \bv|_K\in \bV_R(K)\ \forall K\in \mct\}.
\end{align*}

\begin{theorem}\label{thm:VhRWhRStability}
The pair $\bV_h^R-W_h^R$
is inf-sup stable.
\end{theorem}
\begin{proof}
Let $q\in W_h^R$
and let $\bw\in \bH^1_0(\Omega)$ satisfy $\Div \bw = q$ and $\|\bw\|_{H^1(\Omega)}\le C\|q\|_{L^2(\Omega)}$.
We then define $\bv \in \bV_h^R$ such that
\begin{alignat*}{2}
&\bv(x) = \bI^{SZ}_h \bw(x),\ \Div \bv(x) = q(x),\qquad &&\text{for all vertices } x,\\
&\int_{e} \bv = \int \bI^{SZ}_h \bw\qquad &&\text{for all one-dimensional edges $e$},\\
&\int_{F} \bv\cdot \bn = \int_F \bw\cdot \bn\qquad &&\text{for all $(d-1)$-dimensional faces } F,
\end{alignat*}
where $\bI^{SZ}_h \bw$ is the Scott-Zhang interpolant of $\bw$ \cite{ScottZhang90}.  We then have
$\Div \bv(x) = q(x)$ and
\begin{align*}
\int_K \Div \bv = \int_K \Div \bw = \int_K q.
\end{align*}
Since $q,\Div \bv|_K\in \pol_1^c(K^r)$, we conclude that $\Div \bv = q$.
Uniform inf-sup stability then comes from a standard scaling argument.
\end{proof}

\subsubsection{ Reduced velocity space of  $\bV_h^R$}
In this section, we give a basis for the local space
$\bpol_2(K^r)\cap \bH^1(\Div;K)$, and as a byproduct,
construct reduced spaces of $\bV_h^R$.
To this end, recall that, for a simplex $K\in \mct$, $\lambda_i\in \pol_1^c(K^r)$
satisfy $\lambda_i(x_j) = \delta_{i,j}$.
For each $i\in \{1,\ldots,d+1\}$ we  set 
\begin{equation*}
\bpsi_i= \bc_i \lambda_i^2,
\end{equation*}
where the constant $\bc_i \in \R^d$ is chosen so that 
\begin{equation*}
2 \bc_i \cdot \nab \lambda_i|_{K_j}=1 \text{ for all } 1 \le j \le  d+1, j \neq i.
\end{equation*}
 This is possible since  $\nab \lambda_i|_{K_j}$ for $1 \le j \le  d+1, j \neq i$ are linearly independent. We then see 
 that $\dive \bpsi_i= \lambda_i $, and so $\bpsi_i \in \bpol_2(K^r)\cap \bH^1(\Div;K)$.  
 We note from the proof of Theorem \ref{thm:VRKrDOFsSmooth} that
\begin{equation*}
\dim \bpol_2(K^r)\cap \bH^1(\Div;K)=\dim \bpol_2(K)+ d+1.
\end{equation*}
From this dimension count, we conclude that
\begin{equation*}
\bpol_2(K^r)\cap \bH^1(\Div;K)=\bpol_2(K)+ \text{span} \{ \bpsi_1, \bpsi_2, \ldots, \bpsi_{d+1} \}.
\end{equation*}

Next, using this construction, we reduce the dimension
of $\bV_h^R$ while still getting an inf-sup stable pair.
Recall from Section \ref{sectionBernardiRaugel} 
that $\{\mu_i\}_{i=1}^{d+1}$ are the barycentric coordinates of $K$. 
By the labeling {convention}, we then see that $\mu_i = \lambda_i$ on $\p K$
for all $1\le i\le d+1$. 
We then define
\begin{equation*}
\btheta_i= \frac{1}{2}\bc_i (\lambda_i^2-\mu_i^2),
\end{equation*}
and choose $\bc_i$ so that 
\begin{equation*}
\bc_i \cdot \nab (\lambda_i-\mu_i)|_{K_j}=1 \text{ for all } 1 \le j \le  d+1, j \neq i.
\end{equation*}
We then have
\begin{equation*}
\dive \btheta_i= (\lambda_i-\mu_i)\bc_i \cdot \nab \mu_i+ \lambda_i.
\end{equation*}
In particular, there holds $\dive \btheta_i|_{\p K} = \lambda_i$, 
and since $(\bc_i \cdot \nab \mu_i)$ is constant on $K$, 
$\dive \btheta_i$ is continuous.
Thus, these functions have the following  properties: 
\begin{alignat*}{1}
\btheta_i \in & \bH^1(\Div;K)\cap \bH^1_0(K), \\
\dive \btheta_i(x_j)= &\delta_{ij}.
\end{alignat*}

We  then define the space $\bV^S(K)=\text{span} \{\btheta_1, \btheta_2, \ldots, \btheta_{d+1} \}$. We see that the 
space is a space of bubbles (i.e. vanish on $\partial K$), and that
the degrees of freedom or $\bV^S(K)+ \bV^{\text{MF}}(K)$ are given by 
\begin{alignat*}{1}
&\int_{F_i} \bv \cdot \bn \quad  \text{ for all $(d-1)$-dimensional  faces  $F_i$  of   $K$},  \\
&\dive \bv (x_i) \quad \text{ for all vertices } x_i \text{ of } K.
\end{alignat*}

We can then define
\begin{equation*}
\bV_h^{\dive}=\{ \bv \in \bH^1(\dive;\Omega) \cap \bld{H}_0^1(\Omega): \bv|_K \in \bV^S(K)+ \bV^{\text{MF}}(K), \text{ for all } K \in \mct \}
\end{equation*}
and the degrees of freedom of this space are 
\begin{alignat*}{1}
&\int_{F} \bv \cdot \bn \quad  \text{ for all interior $(d-1)$-dimensional  faces  $F$ of  $\mct$},\\
&\dive \bv (x) \quad \text{ for all vertices $x$ of $\mct$}.
\end{alignat*}

These degrees of freedom give us the 
following result. Its proof is identical to 
the proof of Theorem \ref{thm:VhRWhRStability} and is therefore omitted.

\begin{lemma}\label{lem:divbubble}
It holds, $\dive \bV_h^{\dive} \subset W_h^R$. Moreover, for any $p \in W_h^F$ then there exists a  ${\bv \in} V_h^{\dive}$ so that 
\begin{equation}\label{vpAtl}
\dive \bv=p \qquad \text{ on } \Omega, 
\end{equation}
with the bound
\begin{equation*}
\|\bv\|_{H^1(\Omega)} \le C \| p \|_{L^2(\Omega)}.  
\end{equation*}
\end{lemma}

%

Of course, $\bV_h^{\dive}$ will not have good approximation properties; however,
we can supplement this space with $\mathring{\bpol}_1^c(\mct)$ locally
to obtain a convergent space. The next corollary is immediate.

\begin{corollary}\label{divbubble}
Let
\begin{equation}\label{VHspace}
\bV_h=\{ \bv \in \bld{H^1}(\dive;\Omega) \cap \bld{H}_0^1(\Omega): \bv|_K \in \bpol_1(K)+\bV^S(K)+ \bV^{\text{MF}}(K), \text{ for all } K \in \mct \}.
\end{equation}
Then the pair $\bV_h-W_h^R$ is a divergence-free and inf-sup stable pair.
\end{corollary}
In fact, the degrees of freedom of the space \eqref{VHspace} are
\begin{alignat*}{2}
&\bv(x),\ \dive \bv(x) \quad && \text{ for all interior vertices } x \text{ of } \mct,\\
&\int_{F} \bv \cdot \bn \quad  && \text{ for all interior $(d-1)$-dimensional faces $F$  of  $\mct$}.
\end{alignat*}

Finally, it is clear that  the above space is indeed a subspace  $\bV_h^R$. However, the velocity approximation will converge with one  order less.

\end{document}